\documentclass[10pt]{amsart}
\usepackage{amsfonts,amssymb,amscd,amsmath,enumerate,verbatim,calc}
\newtheorem{thm}{Theorem}[section]
\newtheorem{cor}[thm]{Corollary}

\newtheorem{prop}[thm]{Proposition}

\newtheorem{rem}[thm]{Remark}
\newtheorem{rem and exam}[thm]{Remark and Example}
\newtheorem{rem and exams}[thm]{Remark and Examples}

\DeclareMathOperator{\Ext}{Ext} \DeclareMathOperator{\Supp}{Supp}
\DeclareMathOperator{\V}{V} \DeclareMathOperator{\W}{W}
\DeclareMathOperator{\Hom}{Hom}

\DeclareMathOperator{\lc}{H} 
 
\DeclareMathOperator{\G}{\Gamma} 
 
 \DeclareMathOperator{\Mod}{Mod}

\DeclareMathOperator{\Ass}{Ass}

\newcommand{\fa}{\mathfrak{a}}
\newcommand{\fb}{\mathfrak{b}}

\newcommand{\fm}{\mathfrak{m}}
\newcommand{\fp}{\mathfrak{p}}

\newcommand{\lo}{\longrightarrow}

\begin{document}

\title[Vanishing, Finiteness and Artinianness]
 {Some results on the local cohomology modules
 with respect to a pair of ideals}

\bibliographystyle{amsplain}

   \author[Aghapournahr]{M. Aghapournahr$^{1}$}
\address{$^{1}$ Department of Mathematics, Faculty of Science, Arak University,
Arak, 38156-8-8349, Iran.}
\email{m-aghapour@araku.ac.ir}

   \author[Ahmadi]{Kh. Ahmadi-amoli$^{2}$}
\address{$^{2}$ Department of Mathematics, Payame Noor University, Tehran,
 19395-3697, Iran.}
\email{khahmadi@pnu.ac.ir}

    \author[Sadeghi]{M. Y. Sadeghi$^{3}$}
\address{$^{3}$ Department of Mathematics, Payame Noor University, Tehran,
 19395-3697, Iran.}
\email{m.sadeghi@phd.pnu.ac.ir}



\begin{abstract}
Some affirmative answers are given to
Huneke$^,$s problems. Membership of
$\Hom\big(R/I,$\\$\lc^{i}_{I}(M)\big)$
in an arbitrary Serre class
and so finiteness of
$\Ass\big(\lc^{i}_{I}(M)\big)$ are shown.
Let $\lc^{i}_{I,J}(M)=0$ for all $i<t$ and
$\fa\in\tilde{W}(I,J)$. It is shown that
$\lc^{i}_{\fa}(M)\subseteq\lc^{i}_{I,J}(M)$ for
all $i\leq t$, and so $\lc^{i}_{\fa}(M)=0$ for all
$i<t$. Also, it is proved that if $M$ and
$\lc^{i}_{\fa}(M)$ are finite for all $i<t$, then
for any ideal $\fa\in\tilde{W}(I,J)$ such that $\fa\subseteq\sqrt{(0:\lc^{t}_{I,J}(M))}$,
the $R$-module $\lc^{t}_{\fa}(M)$ is finite.
The calculation of local cohomology
modules with respect to an arbitrary pair of
ideals $I,J$ can be reduced to calculation of
local cohomology modules with respect to a
pair of ideals whose first ideal is generated
by any $k$-regular sequence in $I$.
\end{abstract}

\maketitle

\section{\textbf{Introduction}}
Let $R$ be a commutative Noetherian ring,
$I , J$ be two ideals of $R$, $t\in\mathbb{N}_{0}$, and $M$ be
an $R$-module.
The concept of local cohomology modules with
respect to a pair of ideals introduced by Takahashi,
Yoshino, and Yoshizawa \cite{TakYoYo}
as a generalization of local cohomology modules
on an ideal (the ordinary ones).
As Huneke gave in \cite{Hu} a survey of some
important problems on finiteness, vanishing
and Artinianness of local cohomology modules,
comparing these two kinds of local cohomology
modules is important.

It is known that the finiteness of
$\Ass\Hom\big(R/I,\lc^{i}_{I}(M)\big)$
is one of the sufficient conditions for finiteness of
$\Ass\big(\lc^{i}_{I}(M)\big)$
(\cite{BruHer}, Exercise 1.2.27).
Also, if $\Hom\big(R/I,\lc^{i}_{I}(M)\big)$ is
finite, Artinian, or minimax, then
$\Ass\big(\lc^{i}_{I}(M)\big)$ is finite.
Since the classes of finite, Artinian, and minimax
modules are Serre classes, the
membership of $\Hom\big(R/I,\lc^{i}_{I}(M)\big)$
in a Serre class $\mathcal{S}$ is always useful.

In Section 2,
we give some further contributions to verify the
membership of $\Hom\big(R/I,\lc^{t}_{I}(M)\big)$
in an arbitrary Serre class $\mathcal{S}$
(see for example Proposition \ref{2.1} and
Corollary \ref{2.3}).
As an important result of this section,
we give an affirmative answer to Huneke$^,$s
problem about finiteness of associated primes
of local cohomology modules and also a generalization
of \cite[Exercise 7.1.4 and Theorem 7.1.3]{BroSh}.

In Section 3, we show that vanishing of local
cohomology modules with respect to a pair of
ideals affects the vanishing of local cohomology
modules on an ideal. In fact, if
$\lc^{i}_{I,J}(M)=0$
for all $i<t$, then $\lc^{i}_{\fa}(M)\subseteq
\lc^{i}_{I,J}(M)$ for all $\fa\in\tilde{W}(I,J)$
and all $i\leq t$ (specially for $\fa = I$) and so
$\lc^{i}_{\fa}(M)=0$ for all $\fa\in\tilde{W}(I,J)$
and all $i<t$ (see Theorem \ref{3.4}).
In this situation, if $\fa\in\tilde{W}(I,J)$ be
an ideal such that
$\fa\subseteq\sqrt{(0:\lc^{t}_{I,J}(M))}$, then
$\lc^{t}_{\fa}(M)$ is finite (see Corollary \ref{3.7}).
Huneke$^,$s second problem is about the
finiteness of local cohomology modules.
It is known that if $\fa$ is an ideal of $R$ and
$M$ is a finite $R$-module, a necessary and sufficient
condition for the finiteness of
$\lc^{i}_{\fa}(M)$ for all $i<t$ is $\fa\subseteq\sqrt{(0:\lc^{i}_{\fa}(M))}$ for all $i<t$
(see \cite[Proposition 9.1.2]{BroSh}.
In Corollary \ref{3.13}, we show that if $M$ and
$\lc^{i}_{I,J}(M)$ are finite for all $i<t$, then for an ideal
$\fa\in\tilde{W}(I,J)$ a sufficient condition for the finiteness of
$\lc^{t}_{\fa}(M)$ is that ${\fa}$ satisfies the condition
$\fa\subseteq\sqrt{(0:\lc^{t}_{\fa}(M))}$.
In \cite[Lemma 3.4]{NagSch}, Nagel and Schenzel
showed that for a finite module $M$ over the local
ring $(R,\fm)$ and any positive integer $r$,
$\lc^{i}_{\fm}(M)\cong\lc^{i}_{\fa}(M)$ for all $i<r$,
where $\fa$ is the ideal generated by an $M$-filter
regular sequence $a_1 , a_2 , ... , a_r$.
As the last theorem of this note, we present a generalization of this theorem
in a noetherian ring (not necessary local) for local
cohomology modules with respect to a pair of ideals, by
using the concept of $k$-regular sequences \cite{AhSan}
(see Theorem \ref{3.14}).


\section{\textbf{Finiteness And Artinianness}}

In this section, we study the membership of
$\Hom_R\big(R/\fa,\lc^{t}_{I,J}(M)\big)$ in an arbitrary
Serre class
$\mathcal{S}$ and Artinianness of $\lc^{i}_{\fa}(M)$
for an ideal $\fa$ of $R$.
Recall that a class
$\mathcal{S}$ of $R$-modules is a Serre
subcategory of $R$-modules, if it is
closed under taking submodules, quotients
and extensions.
By ${\Mod}(R)$ and $\mathcal{S}$,
we mean the category
of $R$-modules and $R$-homomorphisms
and an arbitrary Serre subcategory of R-modules, respectively.


\begin{prop} \label {2.1}
Let $\fa\in\tilde{W}(I,J)$ and let $\emph{G : Mod}(R)\rightarrow\emph{Mod}(R)$
be a left exact covariant functor such that $\big( 0 :_X\fa\big) =
\big(0:_{\emph{G}(X)}\fa\big)$ for any $R$-module $X$ and $\emph{G}(E)$
be an injective $R$-module for any injective $R$-module $E$.
Consider the natural homomorphism
$$\psi:\Ext^t_R\big(R/\fa,M\big)\longrightarrow \Hom_R\big(R/\fa,\emph{G}^t(M)\big).$$
\begin{enumerate}
  \item[\rm{(}i\rm{)}] If $\Ext^{t-j}_R\big(R/\fa,\emph{G}^{j}(M)\big)\in\mathcal{S}$
  for all $j<t$, then \emph{Ker} $\psi\in\mathcal{S}$.

  \item[\rm{(}ii\rm{)}] If $\Ext^{t+1-j}_R\big(R/\fa,\emph{G}^{j}(M)\big)\in\mathcal{S}$
  for all $j<t$, then \emph{Coker} $\psi\in\mathcal{S}$.

  \item[\rm{(}iii\rm{)}] If $\Ext^{n-j}_R\big(R/\fa,\emph{G}^{j}(M)\big)\in\mathcal{S}$
  for $n=t,~t+1$ and all $j<t$, then \emph{Ker} $\psi$ and \emph{Coker} $\psi$ both
  belong to $\mathcal{S}$. Thus $\Ext^t_R\big(R/\fa,M\big)\in\mathcal{S}~$  iff
  $~\Hom_R\big(R/\fa,\emph{G}^{t}(M)\big)\in\mathcal{S}$.

\end{enumerate}
\begin{proof}
Let $\textmd{F} (-) = \Hom_R\big(R/\fa,-\big)$.
Then for any $R$-module $X$ we have, $\textmd{FG} (X) = \textmd{F} (X)$.
Now, the assertion follows from \cite[Proposition 3.1]{AghMel1}.
\end{proof}
\end{prop}


From the following proposition which is
obtained by definitions, we get a result
about membership of
$\Hom\big(R/a,\lc^{j}_{I,J}(M)\big)$ in $\mathcal{S}$
followed by some corollaries.

\begin{prop} \label{2.2}
Let $\fa\in\tilde{W}(I,J)$ and $X$
be an $R$-module.
Then $$\big( 0 :_X\fa\big)
= \big(0:_{\Gamma_{\fa}(X)}\fa\big)
=\big(0:_{\Gamma_{\fa,J}(X)}\fa\big)
= \big(0:_{\Gamma_{I,J}(X)}\fa\big).$$

In particular, for any ideal
$\fb$ that contains $I^n$ for some $n$ we have
\begin{enumerate}

  \item[\rm{(}i\rm{)}] $\big( 0 :_X\fb\big)
  = \big(0:_{\Gamma_{\fb}(X)}\fb\big)
  = \big(0:_{\Gamma_{\fb,J}(X)}\fb\big)
  = \big(0:_{\Gamma_{I,J}(X)}\fb\big)
  = \big(0:_{\Gamma_{I}(X)}\fb\big)$,

  \item[\rm{(}ii\rm{)}] $\big(0:_X\fb\big)
  = \big(0:_{\Gamma_{\fb,J}(X)}\fb\big)
  \subseteq \big(0:_{\Gamma_{\fb,J}(X)}I^m\big)
  \subseteq \big(0:_{\Gamma_{I,J}(X)}I^m\big)
  = \big( 0 :_{X}I^m\big)$ for some $m\in\mathbb{N}_0$.
\end{enumerate}
\end{prop}







\begin{cor}\label{2.3}
Let $\fa\in\tilde{W}(I,J)$ be such that
$\Ext^t_R\big(R/\fa,M\big)$,
$\Ext^{n-j}_R\big(R/\fa,
  \lc^{j}_{I,J}(M)\big)\in\mathcal{S}$
  for $n=t,~t+1$ and all $j<t$, then
  $~\Hom_R\big(R/\fa,\lc^{t}_{I,J}(M)\big)\in\mathcal{S}$.

\begin{proof}
Let $\textmd{G}(-)=\G_{I,J}(-)$ and use
Propositions \ref{2.1} and \ref{2.2}.
\end{proof}
\end{cor}


Using various Serre subcategories
of $R$-modules in Propositions \ref{2.1} and \ref{2.2},
we can obtain the following results.


\begin{cor}\label{2.4}
Let $\fa\in\tilde{W}(I,J)$ be such that
$\Ext^{n-j}_R\big(R/\fa,
  \lc^{j}_{I,J}(M)\big)=0$
  for $n=t,~t+1$ and all $j<t$, then
  $\Ext^t_R\big(R/\fa,M\big)\cong
  \Hom_R\big(R/\fa,\lc^{t}_{I,J}(M)\big)$.
\end{cor}


\begin{cor}\label{2.5}
Let $\fa\in\tilde{W}(I,J)$ be such that
$\Ext^{n-j}_R\big(R/\fa,
  \lc^{j}_{I,J}(M)\big)$ is finite
  for $n=t,~t+1$ and all $j<t$, then
  $\Ext^t_R\big(R/\fa,M\big)$ is finite
  if and only if
  $\Hom_R\big(R/\fa,\lc^{t}_{I,J}(M)\big)$
  is finite.
  In particular, if $M$ is finite, then
  $\Hom_R\big(R/\fa,\lc^{t}_{I,J}(M)\big)$
  is finite.
\end{cor}


\begin{cor}\label{2.6}
Let $\fa\in\tilde{W}(I,J)$ be such that
$\Ext^{n-j}_R\big(R/\fa,
  \lc^{j}_{I,J}(M)\big)$ is Artinian
  for $n=t,~t+1$ and all $j<t$, then
  $\Ext^t_R\big(R/\fa,M\big)$ is Artinian
  if and only if
  $\Hom_R\big(R/\fa,\lc^{t}_{I,J}(M)\big)$
  is Artinian.
\end{cor}

\section{\textbf{Vanishing and Anihilator of local cohomology modules}}

In this section we compare vanishing of
local cohomology modules with respect to a
pair of ideals with the vanishing of local
cohomology modules with respect to an ideal.
We begin with the following remark.


\begin{rem} \label{3.1}
\emph{ Let $\big(\Lambda,\leq\big)$ be a
(non-empty) directed partially ordered set.
By an inverse family of ideals of $R$ over
$\Lambda$, we mean a family
$\big({\fb}_\alpha\big)_{\alpha\in\Lambda}$
of ideals of $R$ such that,
whenever $(\alpha,\beta)\in\Lambda\times\Lambda$
with $\alpha\geq\beta$,
we have ${\fb}_\alpha\subseteq{\fb}_\beta$.
For the inverse family of
ideals $\mathcal{B}=\big({\fb}_\alpha\big)_
{\alpha\in\Lambda}$,
we shall denote $\lc^{i}_{\mathcal{B}}(M):
={\varinjlim_{\alpha\in\Lambda}}\Ext^i_R
\big(R/{\fb_\alpha},M\big)$, for all
$i\in\mathbb{N}_0$
(see \cite[1.2.10, 1.2.11, and 2.1.10]{BroSh}).}
\end{rem}


\begin{prop} \label{3.2}
Let $\mathcal{B}=\big({\fb}_\alpha\big)_
{\alpha\in\Lambda}$ be an inverse family
of elements of $\tilde{W}(I,J)$ over
$\Lambda$.
Let $\Ext^{n-j}_R\big(R/{\fb_\alpha},
\lc^{j}_{I,J}(M)\big)=0$ for
any ${\fb_\alpha}\in\mathcal{B}$,
$n=t,~t+1$, and for all $j<t$.
Then $\lc^t_{\mathcal{B}}(M)\subseteq
\lc^t_{I,J}(M)$.
\begin{proof}
By Propositions \ref{2.1} and \ref{2.2},
for $\mathcal{S}=0$ and
$\textmd{G} (-) = \Gamma_{I,J}(-)$,
we have
$\Ext^t_R\big(R/{\fb_\alpha},M\big)\cong\Hom_R
\big(R/{\fb_\alpha},\lc^t_{I,J}(M)\big)$
for any ${\fb_\alpha}\in\mathcal{B}$.
Then, by
\cite[Theorem 1.2.11 and Remarks 1.3.7]{BroSh},
we get
$\lc^t_{\mathcal{B}}(M)\cong\Gamma_{\mathcal{B}}
(\lc^t_{I,J}(M))
\subseteq\lc^t_{I,J}(M)$, as required.
\end{proof}
\end{prop}


\begin{cor} \label{3.3}
Let $\lc^i_{I,J}(M)=0$ for all $i<t$.
Then $\lc^i_{\fb}(M)\subseteq\lc^i_{I,J}(M)$
for any $\fb\in\tilde{W}(I,J)$
and all $i\leq t$, and so $\lc^i_{\fb}(M)=0$
for all $i<t$. In particular,
when $J=0$.
\begin{proof}
Apply Proposition \ref{3.2} to
$\mathcal{B}=\{\fb^n\}_{n\geq1}$ .
\end{proof}
\end{cor}


The next theorem is the other main
theorem in this section which is used
in some results of this paper.
Recall that an $R$-module $M$ is said
to be ZD-module, if for every
submodule $N$ of $M$, the set of
 zero-divisors of $M/N$ is a union of
finitely many prime ideals in Ass$(M/N)$,
(see \cite{Ev, DivEsm}).


\begin{thm} \label {3.4}
Let $\lc^{i}_{I,J}(M)=0$ for all $i<t$.
Then the following statements hold for any
$\fa\in\tilde{W}(I,J)$.
\begin{enumerate}

  \item[\rm{(}i\rm{)}]
  $\Ext^t_R\big(R/\fa,M\big)~\cong~
  {\Hom_R\big(R/\fa,\lc^{t}_{\fa}(M)\big)}$
  \item[] \ \ \ \ \ \ \ \ \ \ \ \ \ \ \
  \ \ \ \
  $~\cong~\Hom_R\big(R/\fa,\lc^{t}_{\fa,J}(M)\big)$
  \item[] \ \ \ \ \ \ \ \ \ \ \ \ \ \ \ \
   \ \ \
  $~\cong~\Hom_R\big(R/\fa,\lc^{t}_{I,J}(M)\big)$.

  \item[\rm{(}ii\rm{)}] $\Gamma_{\fa}
      \big(\lc^{t}_{\fa}(M)\big)\cong\lc^{t}_{\fa}(M)\cong
      \Gamma_{\fa}\big(\lc^{t}_{\fa,J}(M)\big)\cong
      \Gamma_{\fa}\big(\lc^{t}_{I,J}(M)\big)$.

  \item[\rm{(}iii\rm{)}] $\lc^{i}_{\fa}(M)\subseteq
  \lc^{i}_{\fa,J}(M)\subseteq
  \lc^{i}_{I,J}(M)$
    \ \ \ \ \ \ \ \ \ \ \ for all $i\leq t$.

  \item[\rm{(}iv\rm{)}] $\lc^{i}_{\fa,J}(M)
  =\lc^{i}_{\fa}(M)=0$
  \ \ \ \ \ \ for all $i<t$.

  \item[\rm{(}v\rm{)}] ${\Ass}\big(\lc^{i}_{\fa}(M)\big)=
  {\Ass}\big(\lc^{i}_{\fa,J}(M)\big)\cap{\V}(\fa)$
  \item[] \ \ \ \ \ \ \ \ \ \ \ \ \ \ \ \ \
  $={\Ass}\big(\lc^{i}_{I,J}(M)\big)\cap{\V}(\fa)$
  for all $i\leq t$.

 \item[\rm{(}vi\rm{)}] If $\fa\neq0$ and $M$ is a
 \emph{ZD}-module,
  then there exists a regular $M$-sequence of length
  $~t$ contained in $\fa$.

\end{enumerate}
\end{thm}
\begin{proof}

(i) Let $\textmd{G}_1(-)=\Gamma_{\fa}(-)$ ,
$\textmd{G}_2(-)=\Gamma_{\fa,J}(-)$,
$\textmd{G}_3(-)=\Gamma_{I,J}(-)$, and
$\mathcal{S}=0$. Now,
the assertion follows from
Propositions \ref{2.1} and \ref{2.2} and (iv).

(ii) Since $\fa\in\tilde{W}(I,J)$ implies
that ${\fa^n}\in\tilde{W}(I,J)$
for any $n\in\mathbb{N}$, the result follows from (i).

(iii), (iv) Apply Corollary \ref{3.3}
and \cite[Theorem 3.2]{TakYoYo}.

(v) It is obvious by part (i) and
\cite[Exercise 1.2.27]{BruHer}.

(vi) First note that $\lc^{i}_{\fa}(M)=0$
for all $i<t$, by (vi).
By induction on $t$, we construct a regular
$M$-sequence $x_1, x_2, \cdots, x_t$
such that $x_j\in\fa$ for $j=1,2,\cdots,t$.
When $t=0$, there is nothing to prove.
Let $t=1$. Then $\Gamma_{\fa}(M)=0$.
Since $M$ is ZD-module and $\fa$-torsion free,
by the prime Avoidance Theorem, $\fa$ is not
contained in Zd$(M)$.
So, there exists an element $x_1\in\fa$
such that $x_1\not\in$ Zd$(M)$. This proves
the case $t=1$. Now, let $t>1$ and the
assertion be true for $t-1$.
By the above observation, $\fa$ contains an
element $x_1$
which is a non-zerodivisor on $M$. Considering
the exact sequence
$$\lc^{j}_{\fa}(M)\rightarrow\lc^{j}_{\fa}(M/x_1M)
\rightarrow\lc^{j+1}_{\fa}(M),$$
for all $j\in\mathbb{N}_0$, we obtain $\lc^{j}_
{\fa}(M/x_1M)=0$
for all $j<t-1$. Now, by inductive hypothesis,
there is a regular $M/x_1M$-sequence
$x_2, x_3, \cdots , x_t$ such that $x_j\in\fa$
for all $j=2,3,\cdots,t $.
Therefore $x_1, x_2, \cdots , x_t$ is a regular
$M$-sequence.
\end{proof}


It is known that if $\lc^{i}_{\fa}(M)=0$
for all $\fa\in\tilde{W}(I,J)$, then $\lc^{i}_{I,J}(M)=0$
\cite[Theorem 3.2]{TakYoYo}.
In the following we show that the converse of this result is true too, in some situation.

\begin{cor} \label {3.5}
$\lc^{i}_{I,J}(M)=0$ for all $i<t$,
if and only if $\lc^{i}_{\fa}(M)=0$ for any
$\fa\in\tilde{W}(I,J)$
and all $i<t$.
\begin{proof}
The assertion follows easily from Theorem
\ref{3.4} (iv)
 and \cite[Theorem 3.2]{TakYoYo}.
\end{proof}
\end{cor}


\begin{cor} \label {3.6}
Let $M$ be a finite $R$-module such that~
$\lc^{i}_{I,J}(M)=0$ for all $i<t$.
If $\lc^{t}_{I,J}(M)$ is finite,
then $\fa\subseteq\sqrt{(0:\lc^{t}_{\fa}(M))}$
for any $\fa\in\tilde{W}(I,J)$.
\begin{proof}
The assertion follows easily from Theorem
\ref{3.4} (iii).
\end{proof}
\end{cor}


\begin{cor} \label {3.7}
Let $M$ be a finite $R$-module and
$\fa\in\tilde{W}(I,J)$.
Let
$\lc^{i}_{I,J}(M)=0$ for all $i<t$.
If $\fa\subseteq\sqrt{(0:\lc^{t}_{I,J}(M))}$,
then $\fa\subseteq\sqrt{(0:\lc^{t}_{\fa}(M))}$,
and so $\lc^{t}_{\fa}(M)$ is finite.
\begin{proof}
The assertion follows easily from Theorem
\ref{3.4} (iii) and \cite[Proposition 9.1.2]{BroSh}.
\end{proof}
\end{cor}


\begin{cor} \label{3.8}
Let $\lc^{i}_{I,J}(M)=0$ for all $i<t$.
Then for any $\fa\in\tilde{W}(I,J)$
and any $\fb\in\tilde{W}(I,0)$, we have
\begin{enumerate}

  \item[\rm{(}i\rm{)}] $\Hom_R\big(R/I,
  \lc^{t}_{\fa,J}(M)\big)\subseteq
  \Ext^t_R\big(R/I,M\big)$.

  \item[\rm{(}ii\rm{)}] $\Gamma_{I}\big
  (\lc^{t}_{\fa,J}(M)\big)\subseteq
  \lc^{t}_{I}(M)$.

  \item[\rm{(}iii\rm{)}] $\Hom_R\big
  (R/\fb,\lc^{t}_{I}(M)\big)\cong
  \Ext^t_R\big(R/\fb,M\big)$.

  \item[\rm{(}iv\rm{)}] $\Ext^t_R\big
  (R/\fb,M\big)\subseteq
  \Ext^t_R\big(R/I^m,M\big)$ for some
  $m\in\mathbb{N}$. In particular,
  when $\fb\supseteq I$ the result holds
  for all $m\in\mathbb{N}$,
  and so $\Ext^t_R\big(R/\fb,M\big)\subseteq
  \lc^{t}_{I}(M)$.

  \item[\rm{(}v\rm{)}] $\Gamma_{\fb}\big
  (\lc^{t}_{I}(M)\big)\cong
  \lc^{t}_{\fb}(M)$, and so $\lc^{t}_{\fb}
  (M)\subseteq\lc^{t}_{I}(M)$ and ${\Ass}\big(\lc^{t}_
  {\fb}(M)\big)=
  {\Ass}\big(\lc^{t}_{I}(M)\big)\cap{\V}(\fb)$.

\end{enumerate}
\begin{proof}
(i) By Theorem \ref{3.4} (iii),
$\big(0:_{\lc^{t}_{\fa,J}(M)}{I}\big)\subseteq
\big(0:_{\lc^{t}_{I,J}(M)}{I}\big)$.
Now, again apply Theorem \ref{3.4} (i) to $\fa=I$.

(ii) Use part (i).

(iii) Apply Propositions \ref{2.1} and \ref{2.2} for
$\textmd{G}(-)=\Gamma_{I}(-)$ and
$\mathcal{S}=0$.

(iv) Let $m\in\mathbb{N}$ be such that
$I^m\subseteq\fb$.
Then since $\big(0:_{\lc^{t}_{I,J}(M)}
{\fb}\big)\subseteq\big
(0:_{\lc^{t}_{I,J}(M)}{I^m}\big)$,
the assertion follows  from Theorem \ref{3.4} (i).

(v) It is clear by (iii), since
${\fb}^n\in\tilde{W}(I,0)$
for all $n\in\mathbb{N}$.

\end{proof}
\end{cor}


Next, as an application of the above results,
we obtain finiteness
result for $\lc^{t}_{I,J}(M)$ in
Corollary \ref{3.13}. To achieve
this result, we need the following
proposition which easily can be proved by the
same method of \cite[Lemma 3.1]{AghMel2}
for an arbitrary Serre subcategory
$\mathcal{S}$.

\begin{prop}\label{3.9}
Let $\fa$ be an ideal of $R$.
Then $\fa M\in\mathcal{S}$ if and only if
$M/(0:_{M}\fa)\in\mathcal{S}$.

\end{prop}


\begin{cor}\label{3.10}
Let $\fa\in\tilde{W}(I,J)$ be such that
$\Ext^t_R(R/\fa,M)$ and
$\Ext^{n-j}_R\big(R/\fa,\lc^{j}_
{I,J}(M)\big)$ belong to
$\mathcal{S}$ for $n=t,~t+1$, and all $j<t$.
Then $\lc^t_{I,J}(M)\in\mathcal{S}$
if and only if
$\fa\lc^{t}_{I,J}(M)\in\mathcal{S}$.
In particular for $\fa=I$.
\begin{proof}
$(\Rightarrow)$ It is obvious.

$(\Leftarrow)$ The assertion follows
from Proposition \ref{3.9},
Corollary \ref{2.3}
and the short exact sequence
$$0\rightarrow(0:_{\lc^t_{I,J}(M)}\fa)
\rightarrow\lc^t_{I,J}(M)
\rightarrow\lc^t_{I,J}(M)/(0:_{\lc^t_
{I,J}(M)}\fa)\rightarrow0.$$
\end{proof}
\end{cor}


\begin{cor}\label{3.11}
Let $\fa\in\tilde{W}(I,J)$ be such that
$\Ext^t_R(R/\fa,M)$ and
$\Ext^{n-j}_R\big(R/\fa,\lc^{j}_
{I,J}(M)\big)$ are finite
for $n=t,~t+1$, and all $j<t$. Then
$\lc^t_{I,J}(M)$ is finite if and only if
$\fa\lc^{t}_{I,J}(M)$ is finite.

\begin{proof}
Apply Corollary \ref{3.10} to the
class of finite $R$-modules.
\end{proof}
\end{cor}


\begin{cor}\label{3.12}
Let $M$ be a finite $R$-module.
Let $s = \textmd{inf}~\{i\geq0\mid\lc^{i}
_{I,J}(M)~\textmd{is not finite}\}$.
Then ${\fa}^n \lc^{s}_{I,J}(M)$ is not
finite for any $\fa\in\tilde{W}(I,J)$
and all $n\in\mathbb{N}_0$, and so
${\fa}^n \lc^{s}_{I,J}(M)\neq0$.
\begin{proof}
Let $\fa\in\tilde{W}(I,J)$
and $n\in\mathbb{N}_0$. Since
${\fa}^n\in\tilde{W}(I,J)$ and
$\lc^{i}_{I,J}(M)$ is finite, for all
$i<s$, so the assertion follows from
Corollary \ref{3.11}.
\end{proof}
\end{cor}


\begin{cor}\label{3.13}
Let $M$ be a finite $R$-module
and $\lc^{i}_{I,J}(M)$ be finite for all
$i<t$. If there exists $\fa\in\tilde{W}(I,J)$
such that $\fa\subseteq\sqrt
{(0:\lc^{t}_{I,J}(M))}$, then $\lc^{t}_{I,J}(M)$ is finite.
\begin{proof}
Assume that $\lc^{t}_{I,J}(M)$ is not finite.
So, we get
$t = \textmd{inf}~\{i\geq0\mid\lc^{i}
_{I,J}(M)~\textmd{is not finite}\}$.
Now, by using Corollary \ref{3.12},
we get ${\fa}^n \lc^{t}_{I,J}(M)\neq0$
for all $n\in\mathbb{N}_0$, which is a
contradiction with
$\fa\subseteq\sqrt
{(0:\lc^{t}_{I,J}(M))}$.
\end{proof}
\end{cor}


In local case, it is shown that the local
cohomology with respect to the maximal ideal
of $R$ is concerned with the local
cohomology with respect to an ideal generated
by any filter regular sequence.
(See \cite[Lemma 3.4]{NagSch}).
In Theorem \ref{3.14}, as a final result of
this section,
we obtain this in a Noetherian (not necessary local)
ring for a $k$-regular
$M$-sequence $(k\geq-1)$.\\
For a subset $T$ of Spec$(R)$ and an
integer $i\geq-1$,
we set $$(T)_{>i}:=\{\fp\in T\mid
{\dim}(R/\fp)>i\},$$
$$(T)_{\leq i}:=\{\fp\in T\mid
{\dim}(R/\fp)\leq i\}.$$

Recall that, for an integer $k\geq-1$,
a sequence $a_{1},\ldots,a_{n}$
of elements of $R$ is called a
\textit{poor $k$-regular $M$-sequence}
whenever $a_{i}\notin {\fp}$ for all
$${\fp}\in\big({{\Ass}}(M/\sum_{j=1}^{i-1}a_{j}M)\big)_{>k}$$
and all $i=1,\ldots,n.$ Moreover, if
${\dim} (M/\sum^{n}_{i=1}a_{i}M)>k$, then $a_{1},\ldots,a_{n}$ is
called a \textit{$k$-regular $M$-sequence}.
It easy to see that, any poor \textit{$k$-regular $M$-sequence} is a poor
\textit{$(k+1)$-regular $M$-sequence} and for an ideal $I$ of $R$ if $\big(\Supp(M/IM)\big)_{>k+1}\neq\phi$,
then any \textit{$k$-regular $M$-sequence} in $I$ is a \textit{$(k+1)$-regular $M$-sequence}
in $I$.
( To verify for various $k$, see \cite{AhSan}, \cite{LuTa},
\cite{SchTrCu}, \cite{Ah}, \cite{Nh}, and \cite{BroNh}).


\begin{thm} \label {3.14}
Let $a_{1},a_{2},\ldots,a_{n}$ be a
$k$-regular $M$-sequence in $I$ and  set $\fa:=(a_{1},\ldots,a_{n})$.
If $\big(\big({\Supp}(M)\cap{\W}(\fa,J)\big)
\smallsetminus{\W}(I,J)\big)_{\leq{k}}=\emptyset$,
then $\ \lc^{i}_{I,J}(M)\cong{\ \lc^{i}_{\fa,J}(M)}\ $ for all $i<n$.
In particular for $J=0$.
\begin{proof}
Let $$0{\lo}{E^{0}}^{d^{0}\atop{\longrightarrow}}{E^{1}}^{d^{1}\atop
{\longrightarrow}}\cdots\longrightarrow{E^{i}}^{d^{i}\atop{\longrightarrow}}\cdots$$
be a minimal injective resolution for $M$. For all
$i\in{\mathbb{N}_{0}}$ we have
$$E^{i}=\bigoplus_{{\fp}\in\textmd{Supp}(M)}\mu_{i}
({\fp},M)E(R/{\fp})$$
in which $\mu_{i}({\fp},M)$ is
the $i-th$ Bass number of $M$ at the prime ideal ${\fp}$ of
$R$. Let $i<n$ and
${\fp}\in{\big({\Supp}(M)\cap{\W}(\fa,J)\big)_{>k}}.$
By \cite[Theorem 2.3 (ii)]{AhSan}, since $a_{1}/1,\ldots,a_{n}/1$ is a regular
$M_{\fp}$-sequence, we have

\begin{equation}
\mu_{i}({\fp},M)=0.
\end{equation}
Now, by \cite[Proposition 1.11]{TakYoYo}, we have


\begin{equation}
\Gamma_{I,J}(E^{i})~~=\bigoplus_{{\fp}\in{\Supp}(M)
\cap{\W}(I,J)}\mu_{i}({\fp},M)E(R/{\fp}).
\end{equation}
for all $i<n$. Similarly, for the ideal $\fa$, we get

\begin{equation}\Gamma_{\fa,J}(E^{i})~~=\bigoplus_{{\fp}\in
{\Supp}(M)\cap{\W}(\fa,J)}
\mu_{i}({\fp},M)E(R/{\fp}),
\end{equation}
for all $i<n$. On the other hand by
\cite[Proposition 1.6]{TakYoYo},
${\W}(I,J)\subseteq{\W}(\fa,J)$.
So that by (1), (2), (3) and our assumption,
we get $$\Gamma_{I,J}(E^{i})=\Gamma_{\fa,J}(E^{i}),$$ for all
$i<n$. It therefore follows that
$$\lc^{i}_{I,J}(M)=\lc^{i}_{\fa,J}(M),$$for all
$i<n$, as required.
\end{proof}
\end{thm}




\bibliographystyle{amsplain}

\begin{thebibliography}{9}

\bibitem{AghAhSad}
M. Aghapournahr, Kh. Ahmadi-Amoli, and M. Y. Sadeghi,
{\it Upper bounds, cofiniteness and artinianness
of local cohomology modules defined by a pair of ideals},
arXiv:1211.4204v1, [Math.AC] 18-Nov 2012.





\bibitem{AghMel1}
M. Aghapournahr and L. Melkersson,
{\it A natural map in local cohomology},
Ark. Mat.,\textbf{48} (2010), 243--251.


\bibitem{AghMel2}
M. Aghapournahr and L. Melkersson, {\it Finiteness properties of
minimax and coatomic local cohomology modules},
Arch. Math., \textbf{94} (2010), 519--528.




\bibitem{Ah} Kh. Ahmadi-Amoli,
{\it Filter regular sequences, local cohomology
modules and singular sets}.
Ph. D. Thesis, University for
Teacher Education, Iran (1996).





\bibitem{AhSan} Kh. Ahmadi-Amoli and N. Sanaei,
{\it On the $\textit{k}$-regular sequences and
the generalization of $\textit{F}$-modules},
J. Korean Math. Soc. \textbf{49} (2012), (5), 1083--1096.




\bibitem{AsgTo}
M. Asgharzadeh and M. Tousi,
{\it A unified approach to local cohomology modules
using Serre classes},
Canad. Math. Bull. \textbf{53}, (2010), 577--586.






\bibitem{BroNh} M. P. Brodmann and L. T. Nhan,
{\it A finitness result for  associated primes of
certain ext-modules},
Comm. Algebra,
\textbf{36}, (2008), no. 4, 1527--1536.


\bibitem{BroSh}
M. P. Brodmann and R.Y. Sharp,
{\it Local cohomology : an algebraic
introduction with geometric applications},
Cambridge University Press, 1998.


\bibitem{BruHer}
W. Bruns and J. Herzog, {\it Cohen-Macaulay rings},
Cambridge University Press, revised ed., 1998.















\bibitem{DivEsm}
K. Divaani-Aazar and M. A. Esmkhani,
{\it Artinianness of local cohomology modules of ZD-modules},
Comm. Algebra, \textbf{33}, (2005), 2857--2863.


\bibitem{Ev}
E. Graham Evans,
{\it Zero divisors in Noetherian-Like rings},
Trans. Am. Math. Soc, \textbf{155}, (1971), 505--512.










\bibitem{Hu}
C. Huneke,
{\it Problems on local cohomology},
Free Resolutions in commutative algebra and algebraic geometry
(Sundance, Utah,1990), Research Notes in Mathematics 2,
Boston, Ma, Jones and Bartlett Publisher,(1994), 93--108


\bibitem{LuTa}
R. L$\ddot{u}$ and Z. Tang,
{\it The f-depth of an ideal on a module},
Proc. Amer. Math. Soc.
\textbf{130} (2002), no. 7, 1905--1912.
















\bibitem{NagSch}
U. Nagel and P. Schenzel,
{\it Cohomological annihilators and Castelnuovo-Mumford
regularity},
Contemp. Math., \textbf{159},
Amer. Math. Soc., Providence, (1994),
pp.307--328


\bibitem{Nh}
L. T. Nhan,
{\it On generalized regular sequences and the
finiteness for associated primes of local cohomology modules}.
Comm. Alg. \textbf{33} (2005), no. 3, 793--806.








\bibitem{SchTrCu}
P. Schenzel, N. V. Trung, and N. T. Coung
{\it Verallgemeinerte Cohen-Macaulay-Moduln},
Math. Nachr. \textbf{85}, (1978), 57--73.

\bibitem{TakYoYo}
R. Takahashi, Y. Yoshino, and T. Yoshizawa,
{\it Local cohomology based on a
nonclosed support defined by a pair of ideals},
J. Pure. Appl. Algebra. \textbf{213}, (2009), 582--600.










\end{thebibliography}

\end{document}